\documentclass{amsart}
\usepackage{amsmath,amsfonts,amssymb,amsthm,amscd,verbatim}
\usepackage[small,nohug,heads=littlevee]{diagrams}

\begin{document}

\theoremstyle{definition}
\newtheorem{step}{Step}
\newtheorem{definition}{Definition}[section]
\theoremstyle{remark}
\newtheorem{remark}{Remark}[section]
\theoremstyle{plain}
\newtheorem{prop}{Proposition}[section]
\newtheorem{claim}{Claim}[section]
\newtheorem{lemma}{Lemma}[section]
\newtheorem{theorem}{Theorem}[section]
\newtheorem{cor}{Corollary}[section]

\title{Spin$^c$ Prequantization and Symplectic Cutting}
\author{Shay Fuchs}
\maketitle
\begin{abstract}
We define \emph{spin$^c$ prequantization} of a
symplectic manifold to be a spin$^c$ structure
and a connection which are compatible with the
symplectic form. We describe the cutting of an
$S^1$-equivariant spin$^c$ prequantization. The
cutting process involves a choice of a spin$^c$
prequantization for the complex plane. We prove
that the cutting is possible if and only if the
moment map level set along which the cutting is
done is compatible with this choice.
\end{abstract}

\section{Introduction}
Given a compact even-dimensional oriented
Riemannian manifold $M$, endowed with a spin$^c$
structure, one can construct an associated Dirac
operator $D^+$ acting on smooth sections of a
certain (complex) vector bundle over $M$. The
\emph{spin$^c$ quantization} of $M$ with respect
to the above structure is defined to be
$$Q(M)=ker(D^+)-coker(D^+)\ .$$
This is a virtual vector space, and in the
presence of a $G$-action, it is a virtual
representation of the group $G$. Spin$^c$
quantization generalizes the concept of
\emph{K\"{a}hler} and \emph{almost-complex
quantization} (see \cite{CKT}, especially Lemma
2.7 and Remark 2.9) and in some sense it is a
`better behaved' quantization (see \cite{SF1}).

Quantization was originally defined as a process
that associates a Hilbert space to a symplectic
manifold (and self-adjoint operators to smooth
real valued functions on the manifold).
Therefore, one of our goals in this paper is to
relate spin$^c$ quantization to symplectic
geometry. This can be achieved by defining a
\emph{spin$^c$ prequantization} of a symplectic
manifold to be a spin$^c$ structure and a
connection on its determinant line bundle which
are compatible with the symplectic form (in a
certain sense). This definition is analogous to
the definition of prequantization in the context
of geometric quantization (see \cite{GQ} and
references therein). Our definition is different
but equivalent to the one in \cite{CKT}. It is
important to mention that in the equivariant
setting, a spin$^c$ prequantization for a
symplectic manifold $(M,\omega)$ determines a
moment map $\Phi\colon M\to\mathfrak{g}^*$, and
hence the action $G\circlearrowright (M,\omega)$
is Hamiltonian.

The cutting construction was originnaly
introduced by E. Lerman in \cite{L} for
symplecitc manifolds equipped with a Hamiltonian
circle action. In \cite{SF1} we explained how one
can cut a given $S^1$-equivariant spin$^c$
structure on an oriented Riemannian manifold.
Here we extend this construction and describe how
to cut a given $S^1$-equivariant spin$^c$
prequantization. This cutting process involves
two choices: a choice of an equivariant spin$^c$
prequantization for the complex plane $\mathbb
C$, and a choice of a level set
$\Phi^{-1}(\alpha)$ along which the cutting is
done. Our main theorem (Theorem \ref{main-thm})
reveals a quite interesting fact: Those two
choices must be compatible (in a certain sense)
in order to make the cutting construction
possible. In fact, each one of the two choices
determines the other (once we assume that cutting
is possible), so in fact only one choice is to be
made. This theorem also explains the `mysterious'
freedom one has when choosing a spin$^c$
structure on $\mathbb C$ in the first step of the
cutting construction: it is just the freedom of
choosing a `cutting point' $\alpha\in\mathfrak
g^*$ (or a level set of the moment map along
which the cutting is done). Since by our theorem,
$\alpha$ can never be a weight, we see why
spin$^c$ quantization must be additive under
cutting (a result already obtained in
\cite{SF1}).

This paper is organized as follows. In Section
\ref{Sec-preq} we review the definitions of the
spin groups, spin and spin$^c$ structures and
define the concept of spin$^c$ prequantization.
As an example we will use later, we construct a
prequantization for the complex plane. For
technical reasons, we chose to define spin$^c$
prequantization for manifold endowed with closed
two-forms (which may not be symplectic). In
Section \ref{Sec-Cut} we describe the cutting
process in steps and obtain our main theorem
relating the spin$^c$ prequantization for
$\mathbb C$ with the level set used for cutting.
In the last sections we discuss a couple of
examples.

Throughout this paper, all spaces are assumed to
be smooth manifolds, and all maps and actions are
assumed to be smooth. The principal action in a
principal bundle will be always a right action. A
real vector bundle $E$, equipped with a fiberwise
inner product will be called a \emph{Riemannian
vector bundle}. If the fibers are also oriented,
then its bundle of oriented orthonormal frames
will be denoted by $SOF(E)$. For an oriented
Riemannian manifold $M$, we will simply write
$SOF(M)$, instead of $SOF(TM)$.

\textbf{Acknowledgements.} I would like to thank
my supervisor, Yael Karshon, for offering me this
project, guiding and supporting me through the
process of developing and writing the material,
and for having always good advice and a lot of
patience. I also would like to thank Lisa Jeffrey
and Eckhard Meinrenken for useful discussions and
important comments.

\section{Spin$^c$
prequantization}\label{Sec-preq}
\subsection{Spin$^c$ structures}\ \\
In this section we recall the definition and
basic properties of the spin and spin$^c$ groups.
Then we give the definition of a spin$^c$
structure on a manifold, which is essential for
defining spin$^c$ prequantization.

\begin{definition} Let  $V$  be a finite dimensional vector space over  $\mathbb{K}=\mathbb{R}\mbox{ or }
\mathbb{C}$,  equipped with a symmetric bilinear
form  $B:V\times V\rightarrow\mathbb{K}$.  Define
the \emph{Clifford algebra} \ $Cl(V,B)$  to be
the quotient  $T(V)/I(V,B)$  where  $T(V)$  is
the tensor algebra of  $V$,  and  $I(V,B)$  is
the ideal generated by $\{ v\otimes v-B(v,v)\cdot
1\;:\; v\in V\}$.
\end{definition}

\begin{remark}
If $v_1,\dots,v_n$ is an orthogonal basis for
$V$, then $Cl(V,B)$ is the algebra generated by
$v_1,\dots,v_n$, subject to the relations
$v_i^2=B(v_i,v_i)\cdot 1$ and $v_i v_j=-v_j v_i$ for $i\neq j$.\\
Also note that $V$ is a vector subspace of
$Cl(V,B)$.
\end{remark}

\begin{definition}
If $V=\mathbb{R}^k$ and $B$ is minus the standard
inner product on $V$, then define the following
objects:
\begin{enumerate}
\item $C_k=Cl(V,B)$, and $C_k^c=Cl(V,B)\otimes\mathbb{C}$.\\ Those are finite dimensional algebras over $\mathbb{R}$ and $\mathbb{C}$, respectively.
\item The \emph{spin group} $$Spin(k)=\{v_1 v_2 \dots v_l\;:\; v_i\in\mathbb{R}^k,\ ||v_i||=1
\mbox{ and } 0\le l \mbox{ is even}\}\subset
C_k$$
\item The \emph{spin$^c$ group} $$Spin^c(k)= {\left(Spin(k)\times U(1)\right)}\diagup{K}$$ where $U(1)\subset\mathbb{C}$ is the unit circle, and
$K=\{(1,1),(-1,-1)\}$.
\end{enumerate}

\end{definition}

\begin{remark} \
\begin{enumerate}
\item Equivalently, one can define
\begin{multline*}
\qquad \ Spin^c(k)=\\ =\left\{c\cdot v_1\cdots
v_l\;: \linebreak \; v_i\in\mathbb{R}^k,\
||v_i||=1,\  0\le l \mbox{ is even, }\mbox{ and }
c\in U(1)\right\}\subset C^c_k
\end{multline*}
\item The group $Spin(k)$ is connected for $k\ge 2$.
\end{enumerate}
\end{remark}

\begin{prop}
\

\begin{enumerate}
\item There is a linear map $C_k\rightarrow C_k\;,\; x\mapsto x^t$ characterized by
$(v_1\dots v_l)^t=v_l\dots v_1$ for all
$v_1,\dots,v_l\in\mathbb{R}^k$.
\item For each $x\in Spin(k)$ and $y\in\mathbb{R}^k$, we have
$xyx^t\in\mathbb{R}^k$.
\item For each $x\in Spin(k)$, the map
$\lambda(x):\mathbb{R}^k\rightarrow\mathbb{R}^k\;,\;
y\mapsto xyx^t$ is in $SO(k)$, and
$\lambda:Spin(k)\rightarrow SO(k)$ is a double
covering for $k\ge 1$. It is a universal covering
map for $k\ge 3$.
\end{enumerate}
\end{prop}

For the proof, see page 16 in \cite{Fr}.

\begin{definition}
Let $M$ be a manifold, and $Q$ a principal
$SO(k)$-bundle on $M$. A \emph{spin$^c$
structure} on $Q$ is a principal
$Spin^c(k)$-bundle $P\rightarrow M$, together
with a map $\Lambda:P\rightarrow Q$ such that the
following diagram commutes.

$$
\begin{CD}
P\times Spin^c(k)     @>>>          P  \\
@VV\Lambda\times\lambda^c V   @VV\Lambda V \\
Q\times SO(k)       @>>>       Q\\
\end{CD}
$$\\

Here, the maps corresponding to the horizontal
arrows are the principal actions, and
$\lambda^c:Spin^c(k)\rightarrow SO(k)$ is given
by $[x,z]\mapsto\lambda(x)$, where
$\lambda:Spin(k)\rightarrow SO(k)$ is the double
covering.
\end{definition}

\begin{remark}\
\begin{enumerate}
\item  A spin$^c$ structure on an oriented Riemannian vector bundle $E$ is a spin$^c$ structure on the associated bundle of
oriented orthonormal frames,  $SOF(E)$.
\item A spin$^c$ structure on an oriented Riemannian manifold is a
spin$^c$ structure on its tangent bundle.
\end{enumerate}
\end{remark}

\subsection{Equivariant spin$^c$ structures}\ \\
\begin{definition}
Let $G,H$ be Lie groups. A \emph{$G$-equivariant
principal $H$-bundle} is a principal $H$-bundle
$\pi:Q\rightarrow M$ together with left
$G$-actions on $Q$ and $M$, such that:
\begin{enumerate}
\item $\pi(g\cdot q)=g\cdot\pi(q)$ for all $g\in G\;,\; q\in Q$\\
(i.e., $G$ acts on the fiber bundle
$\pi:Q\rightarrow M$).
\item $(g\cdot q)\cdot h=g\cdot(q\cdot h)$ for all $g\in G\;,\; q\in Q\;,\; h\in
H$\\
(i.e., the actions of $G$ and $H$ commute).
\end{enumerate}

\begin{remark}
It is convenient to think of a $G$-equivariant
principal $H$-bundle
in terms of the following commuting diagram (the horizontal arrows correspond to the $G$ and $H$ actions).\\
$$
\begin{CD}
G\times Q       @>>>     Q    @<<<     Q\times H\\
@VId\times\pi VV                   @VV\pi V           @.\\
G\times M       @>>>     M    @.\\
\end{CD}
$$\\
\end{remark}

\begin{definition}
Let $\pi:E\rightarrow M$ be a fiberwise  oriented
Riemannian vector bundle, and let $G$ be a Lie
group. A \emph{$G$-equivariant structure} on $E$
is an action of $G$ on the vector bundle, that
preserves the orientations and the inner products
of the fibers. We will say that $E$ is a
\emph{$G$-equivariant oriented Riemannian vector
bundle}.
\end{definition}

\begin{remark} \
\begin{enumerate}
\item A $G$-equivariant oriented Riemannian vector bundle $E$ over a
manifold $M$, naturally turns $SOF(E)$ into a
$G$-equivariant principal $SO(k)$-bundle, where
$k=rank(E)$.
\item If a Lie group G acts on an oriented Riemannian manifold $M$, by orientation preserving
isometries, then the frame bundle $SOF(M)$
becomes a $G$-equivariant principal
$SO(m)$-bundle, where $m=$dim$(M)$.
\end{enumerate}
\end{remark}

\end{definition}

\begin{definition}
Let $\pi:Q\rightarrow M$ be a $G$-equivariant
principal $SO(k)$-bundle. \emph{A $G$-equivariant
spin$^c$ structure} on $Q$ is a spin$^c$
structure $\Lambda:P\rightarrow Q$ on $Q$,
together with a a left action of $G$ on $P$, such
that
\begin{enumerate}
\item $\Lambda(g\cdot p)=g\cdot\Lambda(p)$ for all $p\in P$, $g\in
G$ (i.e., $G$ acts on the bundle $P\rightarrow
Q$).
\item $g\cdot(p\cdot x)=(g\cdot p)\cdot x$ for all $g\in G$, $p\in P$, $x\in
Spin(k)$\\ (i.e., the actions of $G$ and
$Spin^c(k)$ on $P$ commute).
\end{enumerate}
\end{definition}

\begin{remark}\label{Remark_spin-c_str} \
\begin{enumerate}
\item It is convenient to think of a $G$-equivariant spin$^c$
structure in terms of the following commuting
diagram (where the horizontal arrows correspond
to the principal and the $G$-actions).

$$
\begin{CD}
G\times P       @>>>     P    @<<<     P\times Spin^c(k)\\
@V Id\times\Lambda VV                    @V\Lambda VV          @V\Lambda\times\lambda^c VV\\
G\times Q       @>>>     Q    @<<<     Q\times SO(k)\\
@V Id\times\pi VV                   @V\pi VV           @.\\
G\times M       @>>>     M    @.\\
\end{CD}
$$\\
\item Note that in a $G$-equivariant spin$^c$
structure, the bundle $P\rightarrow M$ is a
$G$-equivariant principal $Spin^c(k)$-bundle.
\end{enumerate}
\end{remark}

\subsection{The definition of spin$^c$
prequantization}\label{def_of_preq}\ \\
In this section we define the concept of \emph{a
$G$-equivariant Spin$^c$ prequantization}. This
will consist of a $G$-equivariant spin$^c$
structure and a connection on the corresponding
$U(1)$-bundle, which is compatible with a given
two-form on our manifold. To motivate the
definition, we begin by proving the following
claim.

\begin{claim}
Let $M$ be a compact oriented Riemannian manifold
of dimension $2m$, on which a Lie group $G$ acts
by orientation preserving isometries, and  let
$P\to SOF(M)\to M$ be a $G$-equivariant spin$^c$
structure on $M$.

Assume that $\theta\colon
TP\to\mathfrak{u}(1)\cong i\mathbb R$ is a
$G$-invariant and $Spin^c(m)$-invariant
connection 1-form on the principal $S^1$-bundle
$\pi\colon P\to SOF(M)$, for which
$$\theta(\zeta_P)\colon P\to
\mathfrak{u}(1)$$
is a constant function for any $\zeta\in\mathfrak{spin}(m)$.\\
 For each
$\xi\in\mathfrak{g}=Lie(G)$ define a map
$$\phi^\xi\colon P\to\mathbb{R}\qquad,\qquad
\phi^\xi=-i\cdot\left(\iota_{\xi_P}\theta\right)\
$$ where $\xi_P$ is the vector field on $P$
generated by $\xi$.\\[10pt]
Then
\begin{enumerate}
\item For any $\xi\in\mathfrak{g}$, the map
$\phi^\xi$ is $Spin^c(2m)$-invariant, i.e.,
$\phi^\xi=\pi^*(\Phi^\xi)$ where $\Phi^\xi\colon
M\to\mathbb{R}$ is a smooth funtion.
\item For any $\xi\in\mathfrak{g}$, we have
$d\Phi^\xi=\iota_{\xi_M}\omega$, where $\omega$
is a real two-form on $M$, determined by the
equation $d\theta=\pi^*(-i\cdot\omega)$.
\item The map $$\Phi\colon
M\to\mathfrak{g}^*\qquad,\qquad
\Phi(m)\xi=\Phi^\xi(m)$$ is $G$-equivariant.
\end{enumerate}
\end{claim}

\begin{proof}\
\begin{enumerate}

\item This follows from the fact that $\theta$ is
$Spin^c(m)$-invariant, and that the $G$ and
$Spin^c(m)$-actions on $P$ commute.
\item For any $\eta=(\zeta,b)\in\mathfrak{spin}^c(m)=\mathfrak{spin}(n)\oplus\mathfrak{u}(1)$, we
have

$$\iota_{\eta_P}\theta=\theta(\eta_P)=\theta(\zeta_P)+\theta(b_P)=\theta(\zeta_P)+b\ .$$

\noindent Since $\theta(\zeta_P)$ is constant by
assumption, we get that

$$\iota_{\eta_P}d\theta=L_{\eta_P}\theta-d\iota_{\eta_P}\theta
=0\ .$$ This implies that $d\theta$ is
horizontal, and hence $\omega$ is well defined by
the equation $d\theta=\pi^*(-i\cdot\omega)$.

Now, observe that
\begin{multline*}
\qquad\qquad\pi^*d\Phi^\xi=d\left(\pi^*\Phi^\xi\right)=d\phi^\xi=-i\;d\iota_{\xi_P}
\theta=-i\left[L_{\xi_P}\theta-\iota_{\xi_P}d\theta\right]=\\
=\iota_{\xi_P}(\pi^*\omega)=\pi^*(\iota_{\xi_M}\omega)\qquad
\end{multline*}
and since $\pi^*$ is injective, we get
$d\Phi^\xi=\iota_{\xi_M}\omega$ as needed.
\item If $g\in G$, $m\in M$, $\xi\in\mathfrak{g}$
and $p\in\pi^{-1}(m)$, then
\begin{multline*}
\qquad\qquad \Phi^\xi(g\cdot m)=\phi^\xi(g\cdot
p)=-i\left(\iota_{\xi_P}\theta\right)(g\cdot p)
=-i\left(\theta_{g\cdot p}(\xi_P|_{g\cdot
p})\right)=\\=-i\left(\theta_{g\cdot
p}(g\cdot(Ad_{g^{-1}}\xi)_P|_p)\right)=
-i\left(\iota
_{\left(Ad_{g^{-1}}\xi\right)_P}\theta\right)(p)=\\
\qquad\qquad\qquad=\phi^{Ad_{g^{-1}}\xi}(p)=\Phi^{Ad_{g^{-1}}\xi}(m)\hfill
\end{multline*}
and we ended up with $\Phi^\xi(g\cdot
m)=\Phi^{Ad_{g^{-1}}\xi}(m)$, which means that
$\Phi$ is $G$-equivariant.
\end{enumerate}
\end{proof}

The above claim suggests a compatibility
condition between a given two-form and a spin$^c$
structure on our manifold. We will work with
two-forms that are closed, but not necessarily
nondegenerate. The compatibility condition is
formulated in the following definition.

\begin{definition}
Let a Lie group $G$ act on a compact
$m$-dimensional manifold $M$, and let $\omega$ be
a $G$-invariant closed two-form (i.e.,
$g^*\omega=\omega$ for any $g\in G$). A
\emph{G-equivariant spin$^c$ prequiantization}
for $M$ is a $G$-equivariant spin$^c$ structure
$\pi\colon P\to SOF(M)\to M$ (with respect to an
invariant Riemannian metric and orientation), and
a $G$ and $Spin^c(m)$-invariant connection
$\theta\in\Omega^1(P;\mathfrak{u}(1))$ on $P\to
SOF(M)$, such that
$$\theta(\zeta_P)=0\text{\ \ \  for any }
\zeta\in\mathfrak{spin}(m)$$ and
$$d\theta=\pi^*(-i\cdot\omega)\ .$$
\end{definition}

\begin{remark}\label{moment map}
By the above claim, the action
$G\circlearrowright (M,\omega)$ is Hamiltonian,
with a moment map $\Phi\colon M\to\mathfrak{g}^*$
satisfying
$$\pi^*\left(\Phi^\xi\right)=-i\cdot\iota_{\xi_P}(\theta)
\mbox{\quad for any \quad}\xi\in\mathfrak{g}\ .$$
\end{remark}

\begin{remark} \label{conn on Spin-c bundle}A $G$-invariant connection
1-form $\theta$ on the $G$-equivariant principal
$Spin^c(m)$-bundle $P\to M$ induces a connection
1-form $\tilde\theta$ on the principal
$S^1$-bundle $P\to SOF(M)$ as follows.

Recall the determinant map
$$det\colon Spin^c(n)\to U(1)\qquad,\qquad
[A,z]\mapsto z^2\ .$$ This map induces a map on
the Lie algebras
$$det_*\colon\mathfrak{spin}^c(n)=\mathfrak{spin}(n)\oplus\mathfrak{u}(1)
\to\mathfrak{u}(1)\simeq i\mathbb{R}\qquad,\qquad
(A,z)\mapsto 2z\ .$$ This means that the map
$\frac{1}{2}det_*\colon\mathfrak{spin}^c(m)\to\mathfrak{u}(1)$
is just the projection onto the $\mathfrak{u}(1)$
component.

The composition $\frac{1}{2}det_*\circ\theta$
will then be a connection 1-form on $P\to
SOF(M)$, which is $G$-invariant, and for which
$\tilde\theta(\zeta_P)=\frac{1}{2}det_*(\zeta)=0$
 for any
$\zeta\in\mathfrak{spin}(m)$.
\end{remark}

\begin{remark}
The condition $\theta(\zeta_P)=0$ could have been
omitted, since our main theorem can be proved
without it. However, this condition is necessary
to obtain a discreet condition on the
prequantizable closed two forms. See the example
in Section \ref{Sec-Ex}.
\end{remark}

In the following claim, $M$ is an oriented
Riemannian $m$-dimensional manifold on which $G$
acts by orientation preserving isometries.

\begin{claim}\label{connection on P_det}
Let $P\to SOF(M)\to M$ be a $G$-equivariant
spin$^c$ structure on $M$. Let
$P_{det}=P/Spin(m)$ and $q\colon P\to P_{det}$
the quotient map. Let $\theta\colon
TP\to\mathfrak u(1)$ be a connection 1-form on
the $G$-equivariant principal $U(1)$-bundle $P\to
SOF(M)$.

Then $\theta=\frac{1}{2}\,q^*(\overline{\theta})$
for some connection one form $\overline\theta$ on
the $G$-equivariant principal $U(1)$ bundle
$P_{det}\to M$ if and only if $\theta$ is
$Spin^c(m)$-invariant and $\theta(\zeta_P)=0$ for
all $\zeta\in\mathfrak{spin}(m)$.
\end{claim}
\noindent Here is the relevant diagram.
$$\begin{CD}
P @>q>> P_{det}\\
@VVV @VVV\\
SOF(M) @>>> M\\
\end{CD}$$
Note that this is not a pullback diagram. The
pullback of $P_{det}$ under the projection
$SOF(M)\to M$ is the square of the principal
$U(1)$ bundle $P\to SOF(M)$.
\begin{proof}[Proof of Claim \ref{connection on
P_det}] Assume that
$\theta=\frac{1}{2}q^*(\overline\theta)$. Then
for any $g\in Spin^c(m)\colon P\to P$, write
$g=[A,z]$ with $A\in Spin(m)$ and $z\in U(1)$.
Since $\theta$ is $U(1)$-invariant, we have
$$g^*\theta=[A,1]^*[1,z]^*\theta=[A,1]^*\theta=
\frac{1}{2}\,[A,1]^*q^*\overline\theta=\frac{1}{2}\,q^*\overline\theta=
\theta\ ,$$ and so $\theta$ is
$Spin^c(m)$-invariant. If
$\zeta\in\mathfrak{spin}(m)$ then
$q_*(\zeta_P)=0$, which implies
$\theta(\zeta_P)=0$.

Conversely, assume that $\theta$ is
$Spin^c(m)$-invariant with $\theta(\zeta_P)=0$
for all $\zeta\in\mathfrak{spin}(m)$. Define a
1-form $TP_{det}\to\mathfrak u(1)$ by
$$\overline\theta(q_*v)=2\,\theta(v)\quad\mbox{for}\quad
v\in TP\ .$$ This will be well defined, since if
$q_*v=q_*v'$ for $v\in T_xP$ and $v'\in
T_{xg}P_{det}$ where $g\in Spin(m)$, then
$q_*(v-v'g^{-1})=0$, which implies that
$v-v'g^{-1}=\zeta_P$ for some
$\zeta\in\mathfrak{spin}(m)$. The fact that
$\theta(\zeta_P)=0$ will imply that
$\theta(v)=\theta(v')$. Smoothness and
$G$-invariance of $\overline\theta$ are straight
forward.

We also need to check that $\overline\theta$ is
vertical (i.e., that
$\overline\theta(\xi_{P_{det}})=\xi$ for
$\xi\in\mathfrak u(1)$). Note that
$Spin^c(m)/Spin(m)$ is isomorphic to $U(1)$ via
the isomorphism taking the class of $[A,z]\in
Spin^c(m)$ to $z^2\in U(1)$. This will imply that
$q_*(\xi_P)=2\,\xi_{P_{det}}$, from which we can
conclude that $\overline\theta$ is vertical.
\end{proof}
\subsection{Spin$^c$ prequantizations for $\mathbb{C}$}\label{prequan. for C}\
\\
For the purpose of cutting, we will need to
choose an $S^1$-equivariant spin$^c$
prequantization on the complex plane. The
$S^1$-action on $\mathbb{C}$ is given by
$$(a,z)\mapsto a^{-1}\cdot z\qquad,\qquad a\in
S^1,\ z\in\mathbb{C}\ .$$ We take the standard
orientation and Riemannian structure on
$\mathbb{C}$ and choose our two-form to be
$$\omega_\mathbb{C}=2\cdot dx\wedge
dy=-i\cdot dz\wedge d\bar z\ .$$

For each odd integer $\ell\in\mathbb{Z}$ we will
define an $S^1$-equivariant spin$^c$
prequantization for $S^1\circlearrowright(\mathbb
C,\omega_\mathbb C)$. The prequantization will be
denoted as
$(P_\mathbb{C}^\ell,\tilde\theta_\mathbb C)$, and
defined as follows.

Let $P_\mathbb{C}^\ell=\mathbb{C}\times
Spin^c(2)$ be the the trivial principal
$Spin^c(2)$-principal bundle over $\mathbb{C}$
with the non-trivial $S^1$-action
$$S^1\times P_\mathbb C^\ell\to P_\mathbb C^\ell\qquad,\qquad(e^{i\varphi}z,(z,[a,w]))\mapsto
(e^{-i\varphi},[x_{-\varphi/2}\cdot
a,e^{-i\ell\varphi/2}\cdot w])$$ where
$x_\varphi=\cos\varphi+\sin\varphi\cdot e_1e_2\in
Spin(2)$. Note that since $\ell\in\mathbb Z$ is
odd, this action is well defined. Next we define
a connection $$\theta_\mathbb C\colon
TP^\ell_\mathbb
C\to\mathfrak{spin}^c(2)=\mathfrak{spin}(2)\oplus\mathfrak
u(1)\ .$$ Denote by $\pi_1\colon P^\ell_\mathbb
C\to \mathbb C $ and $\pi_2\colon P^\ell_\mathbb
C \to Spin^c(2)$ the projections, and by
$\theta^R$ the right-invariant Maurer-Cartan form
on $Spin^c(2)$. Then set
$$\theta_\mathbb C\colon TP_\mathbb C^\ell\to Spin^c(2)
\qquad,\qquad\theta_\mathbb
C=\pi_2^*(\theta^R)+\frac{1}{2}\;\pi_1^*(\bar
z\,dz-z\,d\bar z)\ .$$ Note that $\pi_1^*(\bar
z\,dz-z\,d\bar z)$ takes values in $i\mathbb
R=\mathfrak u(1)\subset\mathfrak{spin}^c(2)$, and
that the connection $\theta_\mathbb C$ does not
depend on $\ell$.\\ Finally, let
$$\tilde\theta_\mathbb{C}=\frac{1}{2}\,det_*\circ\theta_\mathbb{C}\ .$$

\begin{claim}
For any odd $\ell\in\mathbb Z$, the pair
$(P^\ell_\mathbb C,\tilde\theta_\mathbb C)$ is an
$S^1$-equivariant spin$^c$ prequantization for
$(\mathbb C,\omega_\mathbb C)$.
\end{claim}

\begin{proof}
The 1-form $\theta_\mathbb C$ (and hence
$\tilde\theta_\mathbb C$) is $S^1$-invariant,
since $\bar z\,dz-z\,d\bar z$ is an
$S^1$-invariant 1-form on $\mathbb C$, and since
the group $Spin^c(2)$ is abelian. The 1-form
$\tilde\theta_\mathbb C$ is given by
$$\tilde\theta_\mathbb C=\frac{1}{2}\,det_*\circ\theta_\mathbb C=
\frac{1}{2}\,det_*\circ\pi_2^*(\theta^R)+\frac{1}{2}\;\pi_1^*(\bar
z\,dz-z\,d\bar z)$$ and therefore
$$d\left(\tilde\theta_\mathbb
C\right)=0+\frac{1}{2}\,\pi_1^*(d\bar z\wedge
dz-dz\wedge d\bar z)=\pi_1^*\left(-dz\wedge d\bar
z\right)=\pi_1^*(-i\cdot\omega_\mathbb C)$$ as
needed. Finally, by Remark \ref{conn on Spin-c
bundle}, we have $\tilde\theta_\mathbb
C(\zeta_{P_\mathbb C^\ell})=0$ for all
$\zeta\in\mathfrak{spin}(2)$.

\end{proof}

\section{Cutting of a spin$^c$
prequantization}\label{Sec-Cut}

The process cutting consists of several steps:
Taking the product, restricting and taking the
quotient of spin$^c$ structures. We start by
discussing those constructions independently.

\subsection{The product of two spin$^c$
prequantizations}\ \\

Let a Lie group $G$ act by orientation preserving
isometries on two oriented Riemannian manifolds
$M$ and $N$, of dimensions $m$ and $n$,
respectively. Given two equivariant spin$^c$
structures $P_M,P_N$ on $M,N$, we can take their
`product' as follows. First, note that $P_M\times
P_N$ is a $G$-equivariant principal
$Spin^c(m)\times Spin^c(n)$-bundle on $M\times
N$. Second, observe that $Spin^c(m)$ and
$Spin^c(n)$ embed naturally as subgroups of
$Spin^c(m+n)$, and thus give rise to a
homomorphism
$$Spin^c(m)\times Spin^c(n)\to Spin^c(m+n)\qquad,\qquad (x,y)\mapsto x\cdot y\ .$$
This homomorphism is used to define a principal
$Spin^c(m+n)$-bundle on $M\times N$, denoted
$P_{M\times N}$, as a fiber bundle associated to
$P_M\times P_N$.

In the following claim, $\theta^L$ is the left
invariant Maurer-Cartan 1-form on the group
$Spin^c(m+n)$, and $\omega_M,\omega_N$ are closed
$G$-invariant two forms on $M,N$.

\begin{claim}\label{prequan. for products}
Let $(P_M,\theta_M)$ and $(P_N,\theta_N)$ be two
$G$-equivariant spin$^c$ prequantizations for
$(M,\omega_M)$ and $(N,\omega_N)$, respectively.

Let $$P_{M\times N}=\left(P_M\times
P_N\right)\times_{Spin^c(m)\times
Spin^c(n)}Spin^c(m+n)$$ and $$\theta_{M\times
N}=\theta_M+\theta_N+\frac{1}{2}\,det_*\circ\theta^L\in\Omega^1(P_{M\times
N};\mathfrak{u}(1))\ .$$ Then $(P_{M\times
N},\theta_{M\times N})$ is a $G$-equivariant
spin$^c$ prequantization for $(M\times
N,\omega_M\oplus~\omega_N)$, called \emph{the
product of $(P_M,\theta_M)$ and
$(P_N,\theta_N)$}.
\end{claim}

\begin{remark}\label{remark product} \
\begin{enumerate}
\item More specifically, the connection
$\theta_{M\times N}$  is given by
$$\theta_{M\times
N}(q_*(u,v,\xi^L))=\theta_M(u)+\theta_N(v)+\frac{1}{2}\,det_*(\xi)$$
where $u\in TP_M,\ v\in TP_N,\
\xi\in\mathfrak{spin}^c(m+n)$ and $$q\colon
P_M\times P_N\times Spin^c(m+n)\to P_{M\times
N}$$ is the quotient map. This is well defined
since $\theta_M$ and $\theta_N$ are
spin$^c$-invariant.

\item The $G$-action on $M\times N$ can be taken to
be either the diagonal action $$g\cdot
(x,y)=(g\cdot x,g\cdot y)$$ or the `M-action' $$
g\cdot (x,y)=(g\cdot x, y)$$ and $(P_{M\times
N},\theta_{M\times N})$ will be a $G$-equivariant
prequantization with respect to any of those
actions.
\item The map $P_{M\times N}\to SOF(M\times N)$ is the
natural one induced from $P_M\to SOF(M)$ and
$P_N\to SOF(N)$, using the fact that
$$SOF(M\times N)\cong \left(SOF(M)\times
SOF(N)\right)\times_{SO(m)\times SO(n)}SO(m+n)\
.$$
\end{enumerate}
\end{remark}

\begin{proof}
The connection $\theta_{M\times N}$ is $G$ and
$Spin^c(m+n)$-invariant, since $\theta_M$ and
$\theta_N$ have the same invariance properties.
Moreover, since $d\theta^L=0$, we get that
$$d(\theta_{M\times
N})=d(\theta_M)+d(\theta_N)=\pi^*(-i\cdot\omega_M\oplus\omega_N)$$
as needed, where $\pi\colon P_{M\times N}\to
M\times N$ is the projection.\\
Finally, $\theta_{M\times N}(\zeta_{P_{M\times
N}})=0$ for all $\zeta\in\mathfrak{spin}(m+n)$
since $\frac{1}{2}det_*(\zeta)=0$.
\end{proof}

\subsection{Restricting a spin$^c$
prequantization}\label{restr prequan} \ \\
Assume that a Lie group $G$ acts on an $m$
dimensional oriented Riemannian manifold $M$ by
orientation preserving isometries. Let $Z\subset
M$ be a $G$-invariant  co-oriented submanifold of
co-dimension 1. Then there is a natural map
$$i\colon SOF(Z)\to
SOF(M)\qquad,\qquad
i(f)(a_1,\dots,a_m)=f(a_1,\dots,a_{m-1})+a_m\cdot
v_p$$ where
$f\colon\mathbb{R}^{m-1}\xrightarrow{\sim}T_pZ$
is a frame in $SOF(Z)$, and $v\in\Gamma(TM)$ is
the vector field on $Z$ of positive unit vectors
orthogonal to $TZ$.

A $G$-equivariant spin$^c$ structure $P$ on $M$
can be restricted to $Z$, by setting
$$P_Z=i^*(P)\ ,$$
i.e., $P_Z$ is the pullback under $i$ of the
circle bundle $P\to SOF(M)$. The relevant diagram
is
$$
\begin{CD}
P_Z=i^*(P) @>i'>> P \\
@VVV @VVV \\
SOF(Z) @>i>> SOF(M)\\
@VVV @VVV\\
Z @>>> M\\
\end{CD}
$$\\
The principal action on $P_Z\to Z$ comes from the
natural inclusion $Spin^c(m-1)\hookrightarrow
Spin^c(m)$, and the $G$-action on $P_Z$ is
induced from the one on $P$.

Furthermore, if a connection 1-form $\theta$ is
given on the circle bundle $P\to SOF(M)$, we can
restrict it to a connection 1-form
 $\theta_Z$ on $P_Z\to SOF(Z)$ by letting $$\theta_Z=(i')^*\theta\ .$$

\begin{claim}\label{prequan. for Z}
Let $(P,\theta)$ be a $G$-equivariant spin$^c$
prequantization for $(M,\omega)$ (for a closed
$G$-invariant two form $\omega$), and $Z\subset
M$ a co-oriented $G$-invariant submanifold of
co-dimension 1. Then the pair $(P_Z,\theta_Z)$ is
a $G$-equivariant spin$^c$ prequantization for
$(Z,\omega|_Z)$.
\end{claim}

\begin{proof}

$$d(\theta_Z)=(i')^*(d\theta)=
(i')^*\pi^*(-i\cdot\omega)=\pi^*(-i\cdot\omega|_Z)$$
as needed, and
$$\theta_Z(\zeta_{P_Z})=\theta(\zeta_P)=0$$ for all $\zeta\in\mathfrak{spin}(m-1)$.
\end{proof}

\subsection{Quotients of spin$^c$
prequantization} \ \\
Here is a general fact about connections on
principal bundles and their quotients.

\begin{claim}\label{quotient conn.}
Let $H,K,G$ be three Lie groups, and $P\to X$ an
$H$-equivariant and $K$-equivariant principal
$G$-bundle. Assume that $H$ acts freely on $X$,
and that the $H$ and $K$-actions on $P$ commute
(i.e., $h\cdot(k\cdot y)=k\cdot (h\cdot y)$ for
all $h\in H,\ k\in K,\ y\in P$), then:
\begin{enumerate}
\item $\pi\colon P/H\to X/H$ is a $K$-equivariant principal
$G$-bundle.
\item If $\theta\colon TP\to\mathfrak{g}$ is a
connection 1-form, and $q\colon P\to P/H$ is the
quotient map, then $\theta=q^*(\bar\theta)$ for
some connection 1-form $\bar\theta\colon
T(P/H)\to\mathfrak{g}$ if and only if $\theta$ is
$H$-invariant, and $\theta(\xi_P)=0$ for all
$\xi\in\mathfrak{h}$.
\end{enumerate}

\end{claim}

\begin{proof}\
\begin{enumerate}
\item The surjection $P/H\to M/H$, induced from $\pi\colon P\to M$, and the right $G$-action on those
quotient spaces are well defined since the left
$H$-action commutes with the right $G$-action on
$P$, and with the projection $\pi$.

To show that $P/H\to X/H$ is a principal
$G$-bundle, it suffices to check that $G$ acts
freely on $P/H$. Indeed, if $[p]\in P/H,\ g\in G$
and $[p]\cdot g=[p]$, then this implies
$$[p\cdot g]=[p]\qquad\Rightarrow\qquad p\cdot g=h\cdot
p$$ for some $h\in H$, which implies
$$\pi(p\cdot g)=\pi(h\cdot
p)\qquad\Rightarrow\qquad\pi(p)=h\cdot\pi(p)\ .$$
But $H\circlearrowright X$ freely, and so $h=id$.
Then $p\cdot g=p$, and since $P\circlearrowleft
G$ freely, we conclude that $g=id$, as needed.

It is easy to check that the $K$-action descends
to $P/H\to X/H$, since it commutes with the $H$
and the $G$-actions.
\item First assume that $\theta=q^*(\bar\theta)$.
If $h\in H$ acts on $P$, then
$$h^*\theta=h^*(q^*\bar\theta)=(q\circ h)^*\bar\theta=
q^*\bar\theta=\theta$$ and so $\theta$ is
$H$-invariant. Also, if $\xi\in\mathfrak{h}$,
then clearly $q_*(\xi_P)=0$, and hence
$\theta(\xi_P)=(q^*\bar\theta)(\xi_P)=0$, as
needed.

Conversely, assume that $\theta$ is $H$-invariant
and that $\theta(\xi_P)=0$ for all
$\xi\in\mathfrak{h}$. For any $v\in TP$ define
$$\bar\theta(q_*v)=\theta(v)\ .$$
This is well defined: If $v\in T_yP$ and $v'\in
T_{y'}P$ such that $q_*(v)=q_*(v')$, then
$y'=h\cdot y$ for some $h\in H$, and we get that
$$ \theta_{y'}(v')=\theta_{h\cdot y}(v')=h^*(\theta_y((h^{-1})_*v'))=\theta_y((h^{-1})_*v')\ .$$
Now observe that
$$q_*(v-(h^{-1})_*v')=q_*(v)-q_*(v')=0\ ,$$ and so
$v-(h^{-1})_*v'=~\xi_P|_x$ (for some
$\xi\in\mathfrak{h}$) is in the vertical bundle
of $P\to P/H$. By assumption, $\theta(\xi_P)=0$
and therefore $\theta_y(v)=\theta_{y'}(v')$, and
$\bar\theta$ is well defined.

The map $\bar\theta\colon T(P/H)\to\mathfrak g$
is a 1-form. Smoothness is implied from the
definition of the smooth structure on $P/H$. Also
$\bar\theta$ is vertical and $G$-equivariant
because $\theta$ is.
\end{enumerate}
\end{proof}

Now assume that $Z$ is an $n$-dimensional
oriented Riemannian manifold, and $S^1$ acts
freely on $Z$ by isometries. Let $P\to SOF(Z)\to
Z$ be a $G$ and  $S^1$-equivariant spin$^c$
structure on $Z$. We would like to explain how
one can get a $G$-equivariant spin$^c$ structure
on $Z/S^1$, induced from the given one on $Z$.

Denote by $\frac{\partial}{\partial\varphi}\in
Lie(S^1)\simeq i\mathbb{R}$ the generator, and by
$\left(\frac{\partial}{\partial\varphi}\right)_Z$
the corresponding vector field on $Z$. Define the
normal bundle
$$V=\left[\left(\frac{\partial}{\partial\varphi}\right)_Z\right]^\bot\subset TZ$$
and an embedding $\eta\colon SOF(V)\to SOF(Z)$ as
follows. If
$f\colon\mathbb{R}^{n-1}\xrightarrow{\simeq} V_x$
is a frame in $SOF(V)$, then
$\eta(f)\colon\mathbb{R}^{n}\xrightarrow{\simeq}
T_xZ$ will be given by $\eta(f)e_i=f(e_i)$ for
$i=1,\dots,n-1$, and $\eta(f)e_n$ is the unit
vector in the direction of
$\left(\frac{\partial}{\partial\varphi}\right)_{Z,x}$
.

$$
\begin{CD}
\eta^*(P) @>\eta'>> P \\
@VVV @VVV \\
SOF(V) @>\eta>> SOF(Z)\\
@VVV @VVV\\
Z @= Z\\
\end{CD}
$$\\

To get a spin$^c$ structure on $Z/S^1$, first
consider the equivariant spin$^c$ structure on
the vector bundle $V$
$$\eta^*(P)\to SOF(V)\to Z\ .$$
Once we take the quotient by the circle action,
we get \emph{the quotient} spin$^c$ structure on
$Z/S^1$, denoted by $\bar P$:
$$\bar P=\eta^*(P)/S^1\ \to\  SOF(V)/S^1\cong
SOF(Z/S^1)\ \to\  Z/S^1\ \ .$$

If an $S^1$ and $Spin^c(m)$-invariant connection
1-form $\theta$ is given on the principal circle
bundle $P\to SOF(Z)$, then $(\eta')^*\theta$ is a
connection 1-form on the principal  circle bundle
$\eta^*(P)\to SOF(V)$.

The previous claim tells us exactly when the
above connection 1-form will descend to a
connection 1-form on the quotient bundle $\bar
P\to SOF(Z/S^1)$. The following proposition
summarizes the above construction and relates it
to spin$^c$ prequantization.

\begin{prop}\label{main-prop}
Assume that the following data is given:
\begin{enumerate}
\item An $n$-dimensional Riemannian oriented
manifold $Z$.
\item A real closed 2-form $\omega$ on $Z$.
\item Actions of  a Lie group $G$ and $S^1$ on $Z$, by orientation preserving and $\omega$-invariant isometries.
\item A $G$ and $S^1$-equivariant spin$^c$
prequantization $(P,\theta)$ on $Z$. Assume that
the actions of $G$ and $S^1$ on $P$ and $Z$
commute with each other.\\ Also assume that the
action $S^1\circlearrowright Z$ is
free.\vspace{10pt}
\end{enumerate}

 Then, using the above notation, we have
that:
\begin{enumerate}
\item $\theta'=(\eta')^*\theta$ is a connection
1-form on the  principal circle bundle
$\pi\colon\eta^*(P)\to SOF(V)$, satisfying
$$d\theta'=\pi^*(-i\cdot\omega)\ ,$$
and $$\theta'(\zeta_{\eta^*(P)})=0\mbox{\quad for
all \quad}\zeta\in\mathfrak{spin}(m-1)\ .$$
\item If
$\left(\frac{\partial}{\partial\varphi}\right)_{\eta^*(P)}$
is the vector field generated by the action
$S^1\circlearrowright\eta^*(P)$, and $q\colon
\eta^*(P)\to\bar P=\eta^*(P)/S^1$ is the quotient
map, then $\theta'=q^*(\bar\theta)$ for some
connection 1-form $\bar\theta$ on $\bar P\to
SOF(Z/S^1)$ if and only if
$$\theta'\left[\left(\frac{\partial}{\partial\varphi}\right)_{\eta^*(P)}\right]=0\ .$$
Moreover, in this case, $(\bar P,\bar\theta)$ is
a $G$-equivariant spin$^c$ prequantization for
$G\circlearrowright (Z/S^1,\bar\omega)$ (where
$\omega=q^*(\bar\omega)$).
\end{enumerate}

\end{prop}

\begin{proof}\ \\
\begin{enumerate}
\item We have
$$d\theta'=(\eta')^*d\theta=(\eta')^*\circ\pi^*(-i\cdot\omega)=
\pi^*(-i\cdot\omega)$$ and
$$\theta'(\zeta_{\eta^*(P)})=\theta(\zeta_P)=0 $$

as needed.
\item The fact that $\theta'=q^*(\bar\theta)$ if
and only if
$$\theta'\left[\left(\frac{\partial}{\partial\varphi}\right)_{\eta^*(P)}\right]=0$$
follows directly from Claim \ref{quotient conn.},
since $\theta'$ is $S^1$-invariant, and
$\frac{\partial}{\partial\varphi}$ is a
generator.

Finally, $(\bar P,\bar\theta)$ is a
prequantization, since

$$q^*(d{\bar\theta})=d\theta'=
\pi^*(-i\cdot\omega)=q^*\bar\pi^*(-i\cdot\bar\omega)
\quad\Rightarrow\quad
d{\bar\theta}=\bar\pi^*(-i\cdot\bar\omega)$$
where $\bar\pi\colon \eta^*(P)/S^1\to Z/S^1$ is
the projection. Clearly, since all our objects
are $G$-invariant, and all the actions commute,
$(\bar P,\bar\theta)$ is a $G$-equivariant
prequantization.
\end{enumerate}
\end{proof}

\begin{remark} \label{prequan. descends}
When the condition in part (2) of the above
proposition holds, we will say that \emph{the
prequantization $(P,\theta)$ for
$G\circlearrowright (Z,\omega)$ descends to the
prequantization $(\bar P,\bar\theta)$ for
$G\circlearrowright(Z/S^1,\bar\omega)$}.
\end{remark}

\subsection{The cutting of a prequantization} \
\\
In \cite{L}, Lerman describes a cutting
construction for symplectic manifolds
$(M,\omega)$, endowed with a Hamiltonian circle
action and a moment map $\Phi\colon M\to
\mathfrak{u}(1)^*$, which goes as follows. If
$\omega_\mathbb{C}=-i\cdot dz\wedge d\bar z$,
then $(M\times\mathbb
C,\omega\oplus\omega_\mathbb C)$ is a symplectic
manifold. The action $$S^1\times(M\times\mathbb
C)\to M\times\mathbb
C\qquad,\qquad(a,(m,z))\mapsto (a\cdot
m,a^{-1}\cdot z)$$ is Hamiltonian with moment map
$\tilde\Phi(m,z)=\Phi(m)-|z|^2$.

 If
$\alpha\in\mathfrak{u}(1)^*$ and $S^1$ acts
freely on $Z=\Phi^{-1}(\alpha)$, then $\alpha$ is
a regular value of $\tilde\Phi$, and the
(positive) cut space is defined by
$$M_{cut}^+=\tilde\Phi^{-1}(\alpha)/S^1=\left\{(m,z)\in
M\times\mathbb{C}:\Phi(m)-|z|^2=\alpha\right\}\
.$$ This is a symplectic manifold, with the
symplectic form $\omega_{cut}^+$ obtained by
reduction, and $S^1$ acts on $M_{cut}^+$ by
$a\cdot [m,z]=[a\cdot m,z]$. If $M$ is also
Riemannian oriented manifold, so is the cut space
(but the natural inclusion
$M_{cut}^+\hookrightarrow M$ is not an isometry).\\

\noindent Assume that the following is given:
\begin{enumerate}
\item An $m$ dimensional oriented Riemannian
manifold.
\item A closed real two-form $\omega$ on $M$.
\item An action of $S^1$ on
$M$ by $\omega$-invariant isometries.
\item An $S^1$-equivariant spin$^c$
prequantization $(P,\theta)=(P_M,\theta_M)$ for $(M,\omega)$.\\
\end{enumerate}

Recall that the action
$S^1\circlearrowright(M,\omega)$ is Hamiltonian,
with moment map $\Phi\colon
M\to\mathfrak{u}(1)^*$ determined by the equation
$$\pi^*(\Phi^\xi)=-i\cdot\iota_{\xi_P}(\theta)\qquad,\qquad\xi\in\mathfrak{u}(1)$$
where $\pi\colon P\to M$ is the projection,
 and
$\xi_P$ is the vector field on $P$ generated by
the $S^1$-action (see Remark \ref{moment map}).

We want to cut the given spin$^c$
prequantization. For that we choose
$\alpha\in\mathfrak u(1)^*$ and set
$Z=\Phi^{-1}(\alpha)$. We assume that $S^1$ acts
on $Z$ freely, and that $\alpha$ is a regular
value of $\Phi$ (however, we do not assume that
$\omega$ is nondegenerate). Our goal is to get a
condition on $\alpha$ such that cutting along
$Z=\Phi^{-1}(\alpha)$ is possible (i.e., such
that a spin$^c$ prequantization on the cut space
is obtained).\\

We proceed according to the following steps.
\begin{description}
\item[Step 1] Let $S^1$ act on the complex plane
via $$(a,z)\mapsto a^{-1}\cdot z\qquad,\qquad
a\in S^1,\ z\in\mathbb C\ .$$ This action
preserves the standard Riemannian structure and
orientation, and the two form $\omega_\mathbb
C=-i\cdot dz\wedge d\bar z$ .

Fix an odd integer $\ell$, and consider the
$S^1$-equivariant spin$^c$ prequantization
$(P_\mathbb C^\ell,\tilde\theta_\mathbb C)$ for
$S^1\circlearrowright (\mathbb C,\omega_\mathbb
C)$ defined in \S\ref{prequan. for C}.\\

\item[Step 2] Using Claim \ref{prequan. for
products} we obtain an $S^1$-equivariant spin$^c$
prequantization $(P_{M\times\mathbb
C},\theta_{M\times\mathbb C})$ for
$S^1\circlearrowright ({M\times\mathbb
C},\omega\oplus\omega_{\mathbb C})$.\\

\item[Step 3] Denote $$\tilde
Z=\left\{(m,z):\Phi(m)-|z|^2=\alpha\right\}\subset
M\times\mathbb{C}\ .$$ This is an $S^1$-invariant
submanifold of codimension 1. By Claim
\ref{prequan. for Z}, we get an $S^1$-equivariant
spin$^c$ prequantization $(P_{\tilde
Z},\theta_{\tilde Z})$ for $(\tilde
Z,\omega_{\tilde Z})$, where $\omega_{\tilde Z}$
is the restriction of $\omega\oplus\omega_\mathbb
C$ to $\tilde Z$.\\

\item[Step 4] By Remark \ref{remark product}, the
pair $(P_{\tilde Z},\theta_{\tilde Z})$ is an
$S^1$-equivariant prequantization with respect to
both the anti-diagonal and the `M-action' (in
which $S^1$ acts on the $M$ component via the
given action, and on the $\mathbb C$ component
trivially).

Using the terminology introduced in Remark
\ref{prequan. descends}, we state our main
theorem, which enable us to complete the process
and get an equivariant prequantization on the
(positive) cut space.
\end{description}

\begin{theorem} \label{main-thm}
The $S^1$-equivariant spin$^c$ prequantization
$(P_{\tilde Z},\theta_{\tilde Z})$ descends to an
$S^1$-equivatiant spin$^c$ prequantization on
$(\tilde Z/S^1=M_{cut}^+\,,\,\omega_{cut}^+)$ if
and only if
$$\alpha=\frac{\ell}{2}\in\mathfrak{u}(1)^*=\mathbb{R}$$
\end{theorem}

\begin{proof}
By Proposition \ref{main-prop}, $(P_{\tilde
Z},\theta_{\tilde Z})$ will descend to a
prequantization on the cut space, if and only if
$$ \theta_{\tilde
Z}'\left[\left(\frac{\partial}{\partial\varphi}\right)_{\eta^*(P_{\tilde
Z})}\right]=0\ .$$ This is the same as requiring
that $\theta_{\tilde Z}$, when restricted to
$\eta^*(P_{\tilde Z})$, vanishes:
$$\theta_{\tilde Z}\left.\left[\left(\frac{\partial}{\partial\varphi}\right)_{P_{\tilde
Z}}\right]\right|_{\eta^*(P)}=0\ ,$$ which is
equivalent to
$$\theta_{M\times\mathbb
C}\left[\left(\frac{\partial}{\partial\varphi}\right)_{P_{M\times\mathbb
C}}\right]=0\ \ \text{on}\ \ \eta^*(P_{\tilde
Z})\ .$$ Now using the formula for
$\theta_{M\times\mathbb{C}}$, we get that
$${\theta_M}\left(\left(\frac{\partial}{\partial\varphi}\right)_{P_M}\right)+
{\theta_\mathbb
C}\left(\left(\frac{\partial}{\partial\varphi}\right)_{P_\mathbb
C}\right)=0$$ It is not hard to show that at a
point $(z,[A,w])\in P^\ell_\mathbb C=\mathbb
C\times Spin^c(2)$, we have
$$\quad\left(\frac{\partial}{\partial\varphi}\right)_{P^\ell_\mathbb
C}= i\cdot\left[\bar
z\frac{\partial}{\partial\bar
z}-z\frac{\partial}{\partial
z}\right]+\nu|_{[A,w]}$$ where $\nu|_{[A,w]}$ is
the vector field on $Spin^c(2)$ generated by the
element
$$\nu=-\frac{1}{2}\,e_1e_2-\frac{i\cdot \ell}{2}\in\mathfrak{spin}^c(2)\ .$$

Therefore one computes that
\begin{equation*}
\theta_\mathbb
C\left(\left(\frac{\partial}{\partial\varphi}\right)_{P^\ell_\mathbb
C}\right)=-i\cdot\left(|z|^2+\frac{\ell}{2}\right)
\end{equation*}

On the other hand, by the condition defining our
moment map, we have that
$$ {\theta_M}\left(\left(\frac{\partial}{\partial\varphi}\right)_{P_M}\right)
=i\cdot\pi^*\left(\Phi^{{\partial}/{\partial\varphi}}\right)$$
where $\pi\colon P\to M$ is the projection.

Combining the above we see that $(P_{\tilde
Z},\theta_{\tilde Z})$ descends to an
$S^1$-equivatiant spin$^c$ prequantization on
$(\tilde Z/S^1=M_{cut}^+\,,\,\omega_{cut}^+)$ if
and only if (on $\eta^*(P_{\tilde Z})$):
$$\pi^*\left(\Phi^{{\partial}/{\partial\varphi}}\right)-
|z|^2-\frac{\ell}{2}=0\ .$$ But on the manifold
$\tilde Z$ we have $\Phi(m)-|z|^2=\alpha$. and
hence the last equality is equivalent to
$$\alpha-\frac{\ell}{2}=0\
,$$ as needed.
\end{proof}

\begin{remark}
We can also construct a spin$^c$ prequantization
for the negative cut space
$(M_{cut}^-,\omega_{cut}^-)$ as follows. Recall
that $M_{cut}^-$ is defined as the quotient
$$\left\{(m,z)\in M\times\mathbb
C:\Phi(m)+|z|^2=\alpha\right\}/S^1\ ,$$ where the
$S^1$-action on $M\times\mathbb C$ is taken to be
the diagonal action, and $\omega_{cut}^-$ is
defined as before by reduction. The two form on
$\mathbb C$ is taken to be $i\,dz\wedge d\bar z$,
and the spin$^c$ prequantization for $\mathbb C$
is defined using the connection $$\theta_\mathbb
C=\pi^*(\theta^R)-\frac{1}{2}(\bar zdz-zd\bar z)\
.$$The $S^1$-action on $P_\mathbb C^\ell$ will be
given by
$$S^1\times P_\mathbb C^\ell\to P_\mathbb C^\ell\qquad,\qquad(e^{i\varphi}z,(z,[a,w]))\mapsto
(e^{i\varphi},[x_{\varphi/2}\cdot
a,e^{-i\ell\varphi/2}\cdot w])$$ (see
\S\ref{prequan. for C}).

 Other than that, the
construction is carried out as for the positive
cut space, and we can prove a theorem that will
assert that $\alpha=\ell/2$, if the cutting is to
be done along the level set $\Phi^{-1}(\alpha)$
of the moment map.
\end{remark}

\section{An example - The two sphere}\label{Sec-Ex}
In this section we discuss in detail spin$^c$
prequantizations and cutting for the two-sphere.

\subsection{Prequantizations for the two-sphere}
The two-sphere will be thought of as a
submanifold of $\mathbb
R^3$:$$S^2=\{(x,y,z)\in\mathbb
R^3:x^2+y^2+z^2=1\} $$ with the outward
orientation and natural Riemannian structure
induced from the inner product in $\mathbb R^3$.
Fix a real number $c$, and let $\omega=c\cdot A$,
where A is the area form on the two-sphere
$$A=j^*(x\,dy\wedge dz+y\,dz\wedge dx+z\, dx\wedge
dy)\ ,$$ and where $j\colon
S^2\hookrightarrow\mathbb R^3$ is the inclusion.
Note that $\omega$ is a symplectic form if and
only if $c\ne 0$.

For any real $\varphi$ define
\[C_\varphi=\left(
                    \begin{array}{ccc}
                    \cos\varphi & -\sin\varphi & 0\\
                    \sin\varphi & \cos\varphi  & 0\\
                    0          &     0       & 1
                    \end{array}
\right)\ ,\] and let $S^1$ act on $S^2$ via
rotations around the $z$-axis,
i.e.,$$(e^{i\varphi},v)\mapsto C_\varphi\cdot
v\qquad,\qquad v\in S^2\ .$$

In Section 7 of \cite{SF1}, we constructed all
$S^1$-equivariant spin$^c$-structures over the
$S^1$-manifold $S^2$ (up to equivalence). Let us
review the main ingredients here.

First, the \emph{trivial spin$^c$ structure}
$P_0$ is given by the following diagram.

$$
\begin{CD}
S^1\times Spin^c(3)       @>>>     P_0=Spin^c(3)    @<<<     Spin^c(3)\times Spin^c(2)\\
@VVV                    @V\Lambda VV          @VVV\\
S^1\times SO(3)       @>>>     SO(3)    @<<<     SO(3)\times SO(2)\\
@VVV                   @V\pi VV           @.\\
S^1\times S^2       @>>>     S^2    @.\\
\end{CD}
$$\\
In this diagram we use the fact that the frame
bundle of $S^2$ is isomorphic to $SO(3)$. The
projection $\pi$ is given by $$A\mapsto A\cdot
x$$ where $x=(0,0,1)$ is the north pole, and the
map $\Lambda$ is the obvious one.

The horizontal maps describe the $S^1$ and the
principal actions: $S^1$ and $SO(2)$ act on
$SO(3)$ by left and right multiplication by
$C_\varphi$, respectively. The principal action
of $Spin^c(2)$ on $Spin^c(3)$ is just right
multiplication, and the $S^1$ action on
$Spin^c(3)$ is given by
$$(e^{i\varphi},[A,z])\mapsto[x_{\varphi/2}\cdot
A\;,\;e^{i\varphi}\cdot z]$$ where
$x_{\varphi/2}=\cos\varphi+\sin\varphi\cdot
e_1e_2\in Spin(3)$ . We can turn this spin$^c$
structure into a spin$^c$ prequantization as
follows. Let $\omega_0=0$ the zero two form on
$S^2$, and consider the 1-form
$$\theta_0=\frac{1}{2}\;det_*\circ\theta^R\colon
TSpin^c(3)\to \mathfrak u(1)=i\mathbb R$$ where
$\theta^R$ is the right-invariant Maurer-Cartan
form on $Spin^c(3)$ and the map $det$ was defined
in \S\ref{def_of_preq}. Clearly, $(P_0,\theta_0)$
is an $S^1$-equivariant spin$^c$ prequantization
for $(S^2,\omega_0)$.

Next, we construct all $S^1$-equivariant line
bundles over $S^2$.

\begin{claim}
Given a pair of integers $(k,n)$, define an
$S^1$-equivariant complex Hermitian line bundle
$L_{k,n}$ as follows:
\begin{enumerate}
\item As a complex line bundle, $$L_{k,n}=S^3\times_{S^1}\mathbb{C}\ ,$$ where
$S^1$ acts on $\mathbb{C}$ with weight $n$ and on
$S^3\subset\mathbb C^2$ by
$$S^1\times S^3\to S^3\qquad,\qquad(a,(z,w))\mapsto (az,aw)\ .$$
\item The circle group $S^1$ acts on $L_{k,n}$ by
$$ S^1\times L_{k,n}\to L_{k,n}\qquad,\qquad
\left(e^{i\varphi},[(z,w),u]\right)\mapsto
[(e^{i\varphi/2}z,e^{-i\varphi/2}w),e^{i(n+2k)\varphi/2}\cdot
u]\ .$$
\end{enumerate}
Then every equivariant line bundle over $S^2$ is
equivariantly isomorphic to $L_{k,n}$ for some
integers $k,n$.
\end{claim}

For the proof, see Claim 7.1 in \cite{SF1} (where
slightly different notation is used).

To get all spin$^c$ structures on $S^2$, we need
to twist $P_0$ with the $U(1)$-bundle
$U(L_{k,n})$ associated to $L_{k,n}$ for some
$k,n\in\mathbb Z$. Thus define
$$P_{k,n}=P_0\times_{U(1)}U(L_{k,n})\ .$$
The principal $Spin^c(2)$-action is given coming
from the action on $P_0$, and the left
$S^1$-action in induced from the diagonal action.

We now define a connection $$\theta_{n}\colon
TP_{k,n}\to i\mathbb R$$ on the $U(1)$ bundle
$P_{k,n}\to SO(3)=SOF(S^2)$, which will not
depend on $k$, as follows:
$$\theta_{n}=\theta_0+\frac{n}{2}\left(
-\bar z\;dz+z\;d\bar z-\bar w\;dw+w\;d\bar
w\right)+u^{-1}du$$ where $(z,w)\in
S^3\subset\mathbb C^2$ are coordinates on $S^3$
and $u^{-1}du$ is the Maurer-Cartan form on the
$S^1$ component of
$U(L_{k,n})=S^3\times_{S^1}S^1$.

One can compute
$$d\theta_{n}=n(dz\wedge d\bar z+dw\wedge d\bar
w)=\pi^*(-in/2\cdot A)$$ and hence if we define
$\omega_n=\frac{n}{2}\cdot A$ then
$(P_{k,n}\,,\,\theta_n)$ is a spin$^c$
prequantization for $(S^2,\omega_n)$.

Let $P_{det}$ be the $U(1)$-bundle associated to
the determinant line bundle of a spin$^c$
structure. We proved in Section 7 of \cite{SF1},
that the determinant line bundle of any spin$^c$
structure on the two-sphere is isomorphic to
$L_{2k+1,2n}$, and hence has a square root (as a
non-equivariant line bundle). Using this fact and
the construction of $(P_{k,n}\,,\,\theta_n)$
above, we prove:
\begin{claim}\label{S^2-integral}
The $S^1$-manifold $(S^2,\omega=c\cdot A)$ is
spin$^c$-prequantizable (i.e., admits an
$S^1$-equivariant spin$^c$ prequantization) if
and only if $2c\in\mathbb Z$.
\end{claim}

\begin{proof}
Assume that $(P,\theta)$ is a spin$^c$-
prequantization for $(S^2,\omega)$. Then, by
Claim~\ref{connection on P_det},
$\theta=\frac{1}{2}q^*(\overline\theta)$ for some
connection 1-form $\overline\theta$ on the
principal $U(1)$-bundle $p\colon P_{det}\to S^2$,
where $q\colon P\to P/Spin(2)=P_{det}$ is the
quotient map. Since $(P,\theta)$ is a spin$^c$
prequantization, we have $$
d\theta=\pi^*(-i\cdot\omega)\qquad\Rightarrow\qquad
q^*\left(\frac{1}{2}d\overline\theta\right)=q^*p^*(-i\cdot\omega)
\qquad\Rightarrow\qquad
\frac{1}{2}d\overline\theta=p^*(-i\cdot\omega)\
$$
which implies
$$d\overline\theta=p^*(-2i\cdot\omega)\ .$$
This means that $[-2i\cdot\omega]$ is the
curvature class of the determinant line bundle of
$P$. According to the above remark, $P_{det}$ is
a square, and hence the class
$$\frac{1}{2}\,[-2i\cdot\omega]=[-i\cdot\omega]$$
is a curvature class of a line bundle over $S^2$.
This forces $[\omega]$ to be integral (Weyl's
theorem - page 172 in \cite{Fr}), i.e.,
$$\int_{S^2}\omega\in 2\pi\mathbb Z\qquad\Rightarrow
\qquad 2c\in\mathbb Z$$ and the conclusion
follows.

Conversely, assume that $2c\in\mathbb Z$. Then,
as mentioned above, $(P_{k,2c}\,,\,\theta_{2c})$
(for any $k\in\mathbb Z$) is a spin$^c$
prequantization for $(S^2,c\cdot A)$ as needed.
\end{proof}

Let us now compute the moment map $$\Phi\colon
S^2\to\mathfrak u(1)^*=\mathbb R$$ for
$(S^2,n/2\cdot A)$ (for $n\in\mathbb Z$)
determined by the prequantization
$(P_{k,n},\theta_n)$. Recall that
$$\theta_{n}=\theta_0+\frac{n}{2}\left(
-\bar z\;dz+z\;d\bar z-\bar w\;dw+w\;d\bar
w\right)+u^{-1}du\ .$$ It is straightforward to
show that the vector field, generated by the left
$S^1$-action on $P_{k,n}$ is
$$\left(\frac{\partial}{\partial\varphi}\right)_{P_{k,n}}=
\frac{i}{2}\frac{\partial}{\partial
v}-\frac{i}{2}\left(-\bar
z\frac{\partial}{\partial\bar
z}+z\frac{\partial}{\partial z}+\bar
w\frac{\partial}{\partial\bar
w}-w\frac{\partial}{\partial
w}\right)+\frac{i}{2}(n+2k)\frac{\partial}{\partial
u}$$ where $\frac{\partial}{\partial v}$ is the
vector field on $P_0$ generated by the
$S^1$-action.

Now compute
\begin{multline*}
\theta_n\left(\left(\frac{\partial}{\partial\varphi}\right)_{P_{k,n}}\right)=
\frac{i}{2}-\frac{in}{4}\left(-\bar z z+z(-\bar
z)-\bar w w\right)+\frac{i}{2}(n+2k)=\\[5pt]
\qquad=\frac{i}{2}\left[n(|z|^2-|w|^2)+n+2k+1\right]\hfill\\
\end{multline*}
and thus $\Phi$ is given by
$$\Phi([z,w])=-i\cdot\theta_n\left(\left(\frac{\partial}{\partial\varphi}\right)_{P_{k,n}}\right)=
\frac{n}{2}\left(|z|^2-|w|^2+1\right)+k+\frac{1}{2}$$

\begin{remark}
Observe that for $[z,w]\in S^2=\mathbb CP^1$, the
quantity $|z|^2-|w|^2$ represents the third
coordinate $x_3$ (i.e., the height) on the unit
sphere (this is part of the Hopf-fibration).
Since $-1\le x_3\le 1$, we have (for $n\ge0$):
$$k+\frac{1}{2}\le\Phi\le n+k+\frac{1}{2}$$ and hence
the image of the moment map is the closed
interval
$$\left[k+\frac{1}{2}\,,\,n+k+\frac{1}{2}\right]$$
if $n\ge 0$ or
$$\left[n+k+\frac{1}{2}\,,\,k+\frac{1}{2}\right]$$
if $n\le 0$.
\end{remark}

\subsection{Cutting a prequantization on the
two-sphere}

Fix an $S^1$-equivariant spin$^c$-prequantization
$(P_{k,n},\theta_n)$ for $(S^2,\omega_n)$, where
$\omega_n=\frac{n}{2}\cdot A$ ($A$ is the area
form on the two-sphere) and $n\ne 0$.

The corresponding moment map, as computed above,
is $$\Phi\colon S^2\to\mathbb R\qquad,\qquad
\Phi([z,w])=
\frac{n}{2}\left(|z|^2-|w|^2+1\right)+k+\frac{1}{2}$$

We would like to cut this prequantization along a
level set $\Phi^{-1}(\alpha)$ of the moment map.
By Theorem \ref{main-thm} we must have
$$\alpha=\frac{\ell}{2}$$ for some odd integer
$\ell$, and the cutting has to be done using the
spin$^c$ structure $(P^\ell_\mathbb
C,\theta_\mathbb C)$ on $(\mathbb
C,\omega_\mathbb C)$ (see \S\ref{prequan. for
C}).

In \cite[Section 7]{SF1} we performed the cutting
construction for the two-sphere in the case where
$\ell=1$. In this case we showed that the
spin$^c$ structures obtained for the cut spaces
are
$$(P_{k,n})_{cut}^+=P_{0,k+n}\qquad,\qquad
(P_{k,n})_{cut}^-=P_{k,-k}\ .$$ The computations
in \cite{SF1} can be modified for an arbitrary
$\ell$ to get
$$(P_{k,n})_{cut}^+=P_{(\ell-1)/2,k+n-(\ell-1)/2}\qquad,\qquad
(P_{k,n})_{cut}^-=P_{k,-k+(\ell-1)/2}\ .$$

Recall that the cut spaces obtained in this case
are symplectomorphic to two-spheres (if $\ell/2$
is strictly between $k+\frac{1}{2}$ and
$n+k+\frac{1}{2}$). Using this identification we
have:

\begin{claim}
If the symplectic manifold $(S^2,\omega_n)$,
endowed with the Hamiltonian $S^1$-action
$$(e^{i\varphi},v)\mapsto C_\varphi\cdot v$$
and the above moment map $\Phi$ is being cut
along the level set $\Phi^{-1}(\ell/2)$, then the
reduced two-forms on the cut spaces are
$$\omega_{cut}^+=\omega_{k+n+(1-\ell)/2}\qquad and\qquad
\omega_{cut}^-=\omega_{-k+(\ell-1)/2}\ .$$ Here
we assume that $\ell/2$ is strictly between
$k+\frac{1}{2}$ and $n+k+\frac{1}{2}$.
\end{claim}

\begin{proof}
Let us concentrate on the positive cut space. We
will use cylindrical coordinates $(\phi,h)$ to
describe the point
$$(x,y,z)=(\sqrt{1-h^2}\cos\phi,\sqrt{1-h^2}\sin\phi,h)$$
on the unit sphere $S^2$. The positive cut space
is obtained by reduction. The relevant diagram is
$$\begin{CD}
\tilde Z @>i>> S^2\times\mathbb C\\
@V p VV\\
\tilde Z/S^1\cong S^2
\end{CD}$$
Recall that $$\tilde Z=\left\{((\phi,h),u)\in
S^2\times\mathbb
C:\Phi(\phi,h)-|u|^2=\ell/2\right\}$$ and that
the two-form on $S^2\times\mathbb C$ is
$$\omega_n+\omega_\mathbb C=\frac{n}{2}\cdot A
-i\,du\wedge d\bar u\ .$$ The map $p$ is given by
$$((\phi,h),u=r\,e^{-i\alpha})\mapsto
\left(\phi+\alpha\,,\,\frac{2n}{2n+2k+1-\ell}(h-1)+1\right)\
.$$ The pullback of the area form on $S^2$ via
$p$ is
$$A'=(d\phi+d\alpha)\wedge\frac{2n}{2n+2k+1-\ell}dh=
\frac{2n}{2n+2k+1-\ell}(d\phi\wedge
dh-\frac{2i}{n}du\wedge d\bar u)\ ,$$ and thus
the pullback of $\omega_{k+n+(1-\ell)/2}$ via $p$
is
$$\frac{k+n+(1-\ell)/2}{2}\cdot
A'=\frac{n}{2}A-i\,du\wedge d\bar
u=\omega_n+\omega_\mathbb C$$ as needed.

A similar proof is obtained for the negative cut
space.
\end{proof}

To complete the cutting, we need to find out what
are the corresponding connections
$\theta^\pm=(\theta_n)_{cut}^\pm$ on
$(P_{k,n})_{cut}^\pm$. Instead of going through
the cutting process of a connection, we proceed
as follows (for the positive cut space).

We know that
$\left((P_{k,n})_{cut}^+,\theta^+\right)$ must be
a spin$^c$ prequantization for
$$((S^2)_{cut}^+,\omega_{cut}^+)=(S^2,\omega_{k+n+(1-\ell)/2})\ .$$
This means that
$$d\theta^+=d\theta_{k+n+(1-\ell)/2}$$
which implies that
$$\theta^+-\theta_{k+n+(1-\ell)/2}=\pi^*\beta$$
for some closed one-form
$\beta\in\Omega^1(S^2;\mathfrak u(1))$. But then
$\beta=df$ is also exact since $S^2$ is simply
connected. We conclude that
$$\theta^+=\theta_{k+n+(1-\ell)/2}+d(\pi^*(f))\ ,$$
thus, the bundle $((P_{k,n})_{cut}^+,\theta^+)$
is gauge equivalent to
$((P_{k,n})_{cut}^+,\theta_{k+n+(1-\ell)/2})$.

A similar argument can be carried out for the
negative cut space. We summarize:\\
\emph{The cutting of $(S^2,\omega_n)$ along the
level set $\Phi^{-1}(\ell/2)$ yields two spin$^c$
prequantizations:
$$(P_{k,-k+(\ell-1)/2}\,,\,\theta_{-k+(\ell-1)/2})\qquad
\mbox{for}\qquad
((S^2)_{cut}^-=S^2,\omega_{-k+(\ell-1)/2})$$ and
$$(P_{(\ell-1)/2,k+n+(1-\ell)/2}\,,\,\theta_{k+n+(1-\ell)/2})\qquad
\mbox{for}\qquad
((S^2)_{cut}^+=S^2,\omega_{k+n+(1-\ell)/2})\ .$$
}

\section{Prequantizing $\mathbb CP^n$}
In this section we construct a spin$^c$
prequantization for the complex projective space
$\mathbb CP^n$ (with the standard Riemannian
structure coming from the K\"ahler structure).
For $n=1$ we have shown that a two form $\omega$
 on $\mathbb CP^1\cong S^2$ is spin$^c$
prequantizable if and only if
$\frac{1}{2\pi}\omega$ is integral (i.e.,
$\int_{\mathbb
CP^1}\frac{1}{2\pi}\omega\in\mathbb Z$ - see
Claim \ref{S^2-integral}). This is not true in
general. We will prove that for an even $n$, if
$(\mathbb CP^n,\omega)$ is spin$^c$
prequantizable then  $\frac{1}{2\pi}\omega$ will
not be integral. This is an important difference
between spin$^c$ prequantization and the
geometric prequantization scheme of Kostant and
Souriau (an excellent reference for geometric
quantization is \cite{GQ}).

From now on, fix a positive integer $n$. Points
in $\mathbb CP^n$ will be written as $[v]$, where
$v\in S^{2n+1}\subset \mathbb C^{n+1}$. The
Fubini-Study form $\omega_{FS}$ on $\mathbb CP^n$
will be normalized (as in \cite[page
261]{Kirillov}) so that $\int_{\mathbb
CP^1}\omega_{FS}=1$ (where $\mathbb CP^1$ is
naturally embedded into $\mathbb CP^n$). We
describe our construction in steps. For
simpliciy, we discuss the non-equivariant case
(where the acting group $G$ is the trivial
group), but our results will apply to the
equivariant case as well. Also, $|\cdot|$ will
denote the determinant
of a matrix.\\

\noindent \textbf{Step 1 - Constructing a
Spin$^c$
structure.}\\
The group $SU(n+1)$ acts transitively on $\mathbb
CP^n$ via $$SU(n+1)\times\mathbb CP^n\to\mathbb
CP^n\qquad,\qquad (A,[v])\mapsto [A\cdot v]\ .$$
Let $p=e_{n+1}\in\mathbb C^{n+1}$ denote the unit
vector $(0,\dots,0,1)$. The stabilizer of $p$
under the $SU(n+1)$-action is $$H=S(U(n)\times
U(1))=\left\{\left(
               \begin{array}{cc}
                 B & 0 \\
                 0 & |B|^{-1} \\
               \end{array}
             \right)
:B\in U(n)\right\}\subset SU(n+1)$$ and so
$\mathbb CP^n\cong SU(n+1)/H$ via $$[A]\mapsto
[A\cdot p]\ .$$ The tangent space $T_{[p]}\mathbb
CP^n$ can be identified  with $\mathbb C^n$ and
then the isotropy representation is given by
$$\sigma\colon H\to U(n)\qquad,\qquad
\sigma\left(
              \begin{array}{cc}
                B & 0 \\
                0 & |B|^{-1} \\
              \end{array}
            \right)=|B|\cdot B\ .
$$
The frame bundle of $\mathbb CP^n$ can then be
described as an associated bundle (using
$U(n)\subset SO(2n)$):
$$SOF(\mathbb CP^n)=SU(n+1)\times_\sigma SO(2n)\ .$$
The map $$f\colon U(n)\to SO(2n)\times
S^1\qquad,\qquad A\mapsto (A,|A|)$$ has a lift
$F\colon U(n)\to Spin^c(2n)$ (see \cite[page
27]{Fr} for an explicit formula for $F$). Using
that, we define
$$P=SU(n+1)\times_{\tilde\sigma}Spin^c(2n)$$
where $\tilde\sigma=F\circ\sigma\colon H\to
Spin^c(2n)$.

Thus we get a spin$^c$ structure $P\to
SOF(\mathbb CP^n)\to\mathbb CP^n$ on the
n-dimensional complex projective space.\\

\noindent \textbf{Step 2 - Constructing a connection on $P\to SOF(\mathbb CP^n)$ .}\\
Let $\theta^R\colon
TSU(n+1)\to\mathfrak{su}(n+1)$ be the
right-invariant Maurer-Cartan form, and define
$$\chi\colon\mathfrak{su}(n+1)\to\mathfrak
h=Lie(H)\qquad,\qquad \left(
                            \begin{array}{cc}
                              A & \ast \\
                              \ast & -tr(A) \\
                            \end{array}
                          \right)\mapsto
\left(
  \begin{array}{cc}
    A & 0 \\
    0 & -tr(A) \\
  \end{array}
\right)\ .$$ Since $\chi$ is an equivariant map
under the adjoint action of $H$, we conclude that
$$\chi\circ\theta^R\colon TSU(n+1)\to\mathfrak
h$$ is a connection 1-form on the (right-)
principal $H$-bundle $$SU(n+1)\to\mathbb
CP^n=SU(n+1)/H\ .$$ This induces a connection
1-form on the principal $Spin^c(2n)$-bundle
$P\to\mathbb CP^n$:
$$\hat\theta\colon TP\to\mathfrak{spin}^c(2n)\ .$$
After composing $\hat\theta$ with the projection
$$\frac{1}{2}det_*\colon\mathfrak{spin}^c(2n)=\mathfrak{spin}(2n)\oplus
\mathfrak u(1)\to\mathfrak u(1)=i\mathbb R$$ We
get a connection 1-form
$\theta=\frac{1}{2}det_*\circ\hat\theta$ on the
principal $U(1)$-bundle $P\to SOF(\mathbb CP^n)$.

In fact, here is an explicit formula for the
connection $\theta$:\\[5pt]
If $\xi=\left(
          \begin{array}{cc}
            A & \ast \\
            \ast & -tr(A) \\
          \end{array}
        \right)\in\mathfrak{su}(n+1)$,\vspace{5pt}
$\zeta\in\mathfrak{spin}^c(2n)$, $\xi^R$ and
$\zeta^L$ are the corresponding vector fields on
$SU(n+1)$ and $Spin^c(2n)$, and
$$q\colon SU(n+1)\times Spin^c(2n)\to P$$
is the quotient map, then a direct computation
gives
$$\theta(q_*(\xi^R+\zeta^L))=\frac{n+1}{2}\cdot
tr(A)+\frac{1}{2}det_*(\zeta)\ .$$ Note that if
$\zeta\in\mathfrak{spin}(2n)$, then
$\theta(q_*(\zeta^L))=0$.\\

\noindent \textbf{Step 3 - Computing the curvature of $\theta$.}\\
Using the formula
$$d\theta(V,W)=V\,\theta(W)-W\,\theta(V)-\theta([V,W])$$
for any two vector fields $V,W$ on $P$, we can
compute the curvature $d\theta$ of the connection
$\theta$. We obtain the following:\\
If $\xi_1,\xi_2\in\mathfrak{su}(n+1)$,
$\zeta_1,\zeta_2\in\mathfrak{spin}^c(2n)$, and
$$[\xi_1,\xi_2]=\left(
                  \begin{array}{cc}
                    X & \ast \\
                    \ast & \ast \\
                  \end{array}
                \right)\in\mathfrak{su}(n+1)$$
then we have
$$d\theta(q_*(\xi_1^R+\zeta_1^L),q_*(\xi_2^R+\zeta_2^L))=
-\frac{n+1}{2}\cdot tr(X)\ .$$ Let $\omega$ be
the real two form on $\mathbb CP^n$ for which
$$d\theta=\pi^*(-i\cdot\omega)\ .$$

In fact
$$\omega=-\frac{n+1}{2}\cdot 2\pi\,\omega_{FS}
$$ where $\omega_{FS}$ is the Fubini-Study
form. To see this, it is enough, by
$SU(n+1)$-invariance of $\omega$ and
$\omega_{FS}$, to show the above equality at one
point (for instance, at $[p]\in\mathbb CP^n$).

Recall that the cohomology class of $\omega_{FS}$
generates the integral cohomology of $\mathbb
CP^n$, i.e., $\int_{\mathbb CP^1}\omega_{FS}=1$.
This immediately implies that our two form
$\omega$ is integral if
and only if $n$ is odd, and we have:\\
$(P,\theta)$ is a spin$^c$ prequantization for
$(\mathbb CP^n,\omega)$.

\begin{remark}
It is not hard to conclude, that a spin$^c$
prequantizable two form $\omega$ on $\mathbb
CP^n$ is integral if and only if $n$ is odd. In
fact, Proposition D.43 in
\cite{Kar}, together with Claim \ref{connection on P_det} imply the following:\\
\emph{For an odd $n$, a two-form $\omega$ on $\mathbb
CP^n$ is spin$^c$ prequantizable if and only if
$\frac{1}{2\pi}\omega$ is integral, i.e.,
$\left[\frac{1}{2\pi}\omega\right]\in\mathbb
Z[\omega_{FS}]$.\\
For an even $n$, a two-form $\omega$ on $\mathbb
CP^n$ is spin$^c$ prequantizable if and only if
$\left[\frac{1}{2\pi}\omega\right]\in\left(\mathbb
Z+\frac{1}{2}\right)[\omega_{FS}]$.}
\end{remark}


\end{document}